\documentclass[preprint,12pt]{elsarticle}
\usepackage{amsmath,amssymb,amsthm}
\usepackage{xcolor}
\usepackage{graphicx}
\usepackage{geometry}
\usepackage{caption}
\usepackage{hyperref}
\usepackage{cases}
\usepackage{multirow}

\newtheorem{theorem}{Theorem}[section]
\newtheorem{lemma}[theorem]{Lemma}

\newtheorem{remark}{Remark}

\DeclareMathOperator*{\esup}{ess\ sup}

\journal{...}

\begin{document}

\begin{frontmatter}

\title{On the geometric Brownian motion  with state-dependent \\  variable exponent diffusion term}
\author{Mustafa Avci}

\affiliation{organization={Faculty of Science and Technology, Applied Mathematics\\  Athabasca University},
            city={Athabasca},
            postcode={T9S 3A3},
            state={AB},
            country={Canada}}
\ead{mavci@athabascau.ca (primary) & avcixmustafa@gmail.com}

\begin{abstract}
We propose a new  stochastic model involving state-dependent variable exponent $p(\cdot)$ which allows modeling of systems where noise intensity adapts to the current state. This new flexible theoretical framework generalizes both the geometric Brownian motion (GBM) and the Constant-Elasticity-of-Variance (CEV) models. We prove an existence-uniqueness theorem. We obtain an upper-bound approximation for the model-to-model pathwise error between our model and the GBM model as well as test its accuracy through analytical and numerical error estimates. A detailed comparison of the It\^o and Stratonovich interpretations for the proposed model is presented in the Appendix.
\end{abstract}


\end{frontmatter}

\section{Introduction}\label{Sec1}
We study the existence and uniqueness of the solutions for the following nonlinear stochastic differential equation with state-dependent variable exponent $p(\cdot)$ in the diffusion
\begin{align}\label{eq.1}
dX(t) &= \mu X(t)dt + \sigma X(t)^{p(X(t))}dW(t),\,\, t\in [0,T] \nonumber \\
X(0)&=x_0, \tag{{${\mathcal{P}}$}}
\end{align}
where the initial condition $X(0) = x_0$ is a random variable; $W(t):=\{W(t),\, t\geq 0\}$ is a standard Brownian motion; $p(\cdot)$ is a real-valued function of the state variable $X(t)$ with the state space $(0,\infty)$; and $\mu ,\sigma> 0$ are real parameters.\\

The proposed model, \eqref{eq.1}, is a generalization of the Constant-Elasticity-of-Variance (CEV) model (excluding the case $0\leq \gamma<1$)
\begin{align} \label{eq.1a}
dX(t) &=\mu X(t))dt + \sigma X(t)^{\gamma}dW(t),\, \gamma \geq 0,\, \mu,\sigma> 0
\end{align}
and the Geometric Brownian Motion (GBM) model
\begin{align} \label{eq.1b}
dX(t) &=\mu X(t))dt + \sigma X(t)dW(t).
\end{align}
The GBM model, a foundational tool in modern finance  \cite{bs1973}, is defined by a constant log-volatility.While mathematically tractable, this assumption fails to capture empirically observed phenomena like the leverage effect or volatility smiles. Our model addresses this limitation directly, the state-dependent variable exponent function $p(\cdot)$  allows volatility to adapt to the asset's state, providing significant  flexibility.\\
The CEV model \cite{cr76} improves upon the GBM model by introducing a constant exponent $\gamma$, creating a static power-law relationship between volatility and price. However, a key weakness of the CEV model is that for $\gamma>1$  and a positive drift $\mu$, the second moment can blow up in finite time, whereas our hypotheses $(\mathbf{p_1})$-$(\mathbf{p_3})$ force linear growth and thus ensures the existence of a finite second moment for the solution for all time.\\
In our model, we keep the GBM’s linear drift but replace the diffusion $x \to x$ by the state-dependent elasticity $x \to x^{p(x)}$. We define a class of variable exponent functions, $\mathcal{S}$, satisfying the hypotheses $(\mathbf{p_1})$-$(\mathbf{p_3})$. Under the hypotheses $(\mathbf{p_1})$-$(\mathbf{p_3})$, we prove Lipschitz and a linear growth conditions for $x^{p(x)}$, which yield pathwise-unique strong solutions with continuous paths and an a-priori $L^2$ bound. Moreover, the model  can be efficiently computed using standard discretization techniques, such as the log-transformed Milstein or Euler-Maruyama schemes, which ensure both numerical stability and efficiency. With the state-dependent structure, we believe that our model can be used for modeling various real-world phenomena such as volatility clustering in finance, density-dependent noise in biology, and heterogeneous diffusion in physics.

\section{Preliminaries and auxiliary results} \label{Sec2}
We start with some basic concepts of the measure-theoretic probability and the theory of stochastic process. For more details, we refer the reader to monographs \cite{bz98,ovidiu22,friedman75,mao07,oksendal03}.\\
Let $W(t)$ be a Brownian motion on a probability space $(\Omega, \mathcal{F}, \mathbb{P})$ under the real-world measure $\mathbb{P}$. Let $\mathcal{F}_t$, $t\geq 0$, be an increasing family of $\sigma$-fields, called a \emph{filtration}, denoted by $\{\mathcal{F}_t\}_{t \geq 0}$. Then $(\Omega, \mathcal{F}, \{\mathcal{F}_t\}_{t \geq 0}, \mathbb{P})$ is called a \emph{filtered probability space}.\\
A \emph{stochastic process} is an indexed collection of random variables $X(t):=\{X(t):\, a\leq t \leq b\}$, with $0\leq a < b< \infty$. $X(t)$ is said to be \emph{adapted} to $\mathcal{F}_{t}$ if for each $t \in [a,b]$, $X(t)$ is $\mathcal{F}_{t}$-measurable.\\
We denote by $\mathcal{L}^{q}[a,b]$ ($1\leq q \leq \infty$) the class of real-valued $\mathcal{F}_{t}$-adapted stochastic processes $X(t)$  satisfying
\begin{align*}
\mathbb{P}\left\{\int_{a}^{b}|X(t)|^{q}dt<\infty \right\}&=1 \, \text{ if }\, 1\leq q < \infty,\,
\mathbb{P}\left\{\esup_{a \leq t \leq b}|X(t)|<\infty \right\}&=1\, \text{ if }\, q=\infty.
\end{align*}
By $\mathcal{M}^{q}[a,b]$ we denote the subset of $\mathcal{L}^{q}[a,b]$ consisting of all stochastic processes $X(t)$ satisfying
\begin{equation*}
\mathbb{E}\left[\int_{a}^{b}|X(t)|^{q}dt\right]<\infty \, \text{ if }\, 1\leq q < \infty \quad \text{and} \quad \mathbb{E}\left[\esup_{a \leq t \leq b}|X(t)| \right]<\infty\, \text{ if }\, q=\infty,
\end{equation*}
Consider the $1$-dimensional stochastic differential equation of It\^{o} type
\begin{equation}\label{eq.2}
dX(t) = \mu(X(t),t)dt + \sigma(X(t),t)dW(t), \quad t\in[t_0,T],
\end{equation}
with the initial value $X(t_0)=X_0 \in \mathbb{R}$ is an random variable. The equation \eqref{eq.2} is equivalent to the following stochastic integral equation
\begin{equation}\label{eq.3}
X(t) = X_0 + \int_{t_0}^{t} \mu(X(s),s)ds + \int_{t_0}^{t} \sigma(X(s),s)dW(s), \quad t\in[t_0,T],
\end{equation}
where $\mu:\mathbb{R}\times[t_0,T]\to \mathbb{R}$ and $\sigma:\mathbb{R}\times[t_0,T]\to \mathbb{R}$ are both Borel measurable functions.\\

We define a class of variable exponent functions, denoted by $\mathcal{S}$. We say that the function $p(\cdot)$ belongs to the class $\mathcal{S}$ if it satisfies the following hypotheses:
\begin{itemize}
   \item [($\mathbf{p_1}$)] $p(\cdot): (0, \infty) \to \mathbb{R}$ is a differentiable function satisfying
   \begin{align*}
     1\leq p^-:=\inf_{x> 0} p(x)\quad \text{and} \quad \sup_{x> 0} p(x):=p^+ <\infty.
   \end{align*}
   \item [($\mathbf{p_2}$)]It holds
   \begin{align*}
    \lim_{x \to \infty}p(x)=1 \quad \text{and} \quad \limsup_{x \to \infty}(p(x)-1)\log(x)<\infty.
   \end{align*}
   \item [($\mathbf{p_3}$)] There exist real numbers $\delta, M_0, C_0>0$ and $\alpha>0$ with $p^+<1+\alpha$ such that
   \begin{equation*}
   |p^{\prime}(x)|\leq
   \left\{ \begin{array}{ll}
   M_0,\quad & 0<x\leq \delta, \\
   C_0 x^{-1-\alpha},\quad & x > \delta.
   \end{array}\right.
   \end{equation*}
\end{itemize}
\begin{remark}\label{Rem:1} An example of $p(\cdot)$ which satisfies the hypotheses $(\mathbf{p_1})$-$(\mathbf{p_3})$ is $p(x)=1+\frac{a}{(1+x)^2}$, $1 \leq a\leq p^+-1$. Indeed, $(\mathbf{p_1})$ and $(\mathbf{p_2})$ easily follows. As for $(\mathbf{p_3})$, if one lets $\delta \leq 1$, $M_0=\frac{2a}{\delta}$, $C_0= 2a$ and $\alpha=2$ then
  \begin{equation*}
   |p^{\prime}(x)|\leq
   \left\{\begin{array}{ll}
   \frac{2a}{\delta},\quad & 0<x\leq 1, \\
   2a x^{-3},\quad & x > 1.
   \end{array}\right.
   \end{equation*}
\end{remark}

\begin{lemma}\label{Lem:2.1}
Assume that $ x_0 > 0$, and $(\mathbf{p_1})$, $(\mathbf{p_3})$ hold. Then the process $X(t)$ is strictly positive for $t \in [0,T]$.
\end{lemma}

\begin{proof}
We use the following Feller’s test \cite{feller52,feller51}, which is also know as the local (differential) form of Feller’s non-attainability test, to analyze the behavior of the diffusion processes $X(t)$ at the boundary (i.e. $X(t)=0$) of its state space $(0,\infty)$.\\
Consider equation \eqref{eq.2}. If the condition
\begin{equation}\label{eq.4}
\lim_{x \to 0}\left(\mu(x)-\frac{1}{2}\frac{\partial  \sigma^2(x)}{\partial x} \right)\geq 0
\end{equation}
holds, then the boundary $X(t)=0$ is non-attainable for the process $X(t)$ if  $X(0)>0$.\\
Now, we apply \eqref{eq.4} to the diffusion process \eqref{eq.1}.\\
Notice that since $x^{p(x)}=e^{p(x)\log(x)}$, as a composition of differentiable functions, $x^{p(x)}$ is differentiable on $(0,\infty)$ with the derivative
\begin{align}\label{eq.5a}
\frac{d x^{p(x)}}{dx}=p(x)x^{p(x)-1}+p'(x)x^{p(x)}\log(x).
\end{align}
Now define
\begin{equation}\label{eq.5}
\mathcal{T}_p(x)=\mu x-\frac{\sigma^2}{2}\frac{\partial x^{2p(x)}}{\partial x}.
\end{equation}
Then
\begin{align}\label{eq.6}
\mathcal{T}_{p}(x)\geq \mu x-\sigma^2 \left(M_0x^{2p(x)}|\log(x)|+x^{2p(x)-1}p^+\right),
\end{align}
for $0<x\leq\delta<1$. Now applying the elementary limit theorems and the assumptions we have, it reads
\begin{align}\label{eq.7}
\lim_{x \to 0^+}\mathcal{T}_{p}(x)& \geq \lim_{x \to 0^+} \left(\mu x-\sigma^2 \left(M_0x^{2p(x)}|\log(x)|+x^{2p(x)-1}p^+\right)\right)= 0,
\end{align}
which is the desired. In conclusion, the diffusion process $X(t)$ of \eqref{eq.1} started in $(0,\infty)$ never hits zero almost surely (a.s.), and hence, \eqref{eq.1} is well-defined.
\end{proof}

\section{Existence and uniqueness results} \label{Sec3}

\begin{lemma}\label{Lem:2.2} Assume that $p(\cdot) \in \mathcal{S}$. Then the following statements hold true.
\begin{itemize}
  \item [$(i)$] There exists a constant $L>0$ such that
  \begin{equation}\label{eq.7a}
     |x^{p(x)}-y^{p(y)}| \leq L|x-y|, \quad \forall x,y \in (0,\infty).
  \end{equation}
  \item [$(ii)$] There exists a constant $K>0$ such that
  \begin{equation}\label{eq.7b}
   x^{p(x)}\leq K(1+x), \quad \forall x\in (0,\infty).
\end{equation}
\end{itemize}
\end{lemma}

\begin{proof} $(i)$ Let $\varphi(x):=x^{p(x)}$. There are two cases to analyse.\\
\textbf{Case} $0<x\leq\delta\leq1$: First note that $x^{p(x)}\left|\log(x)\right| \to 0$ as $x \to 0^+$. Next, since $p(\cdot)$ is differentiable on  $(0, \infty)$, we have
$\lim_{x \to 0^+ } p(x)=p(0^+)\geq 1$.\\
\textit{If} $p(0^+) > 1$; since $(p(x)-1)\log(x) \to -\infty$ as $x \to 0^+$, we obtain that \newline $\lim_{x \to 0^+ }x^{p(x)-1} =0$.\\
\textit{If} $p(0^+) = 1$; applying the Mean Value Theorem (MVT) on the closed interval $[0^+,x] \subset (0,\delta)$ for an inner point $\xi \in (0^+,x)$ gives
\begin{equation}\label{eq.8}
\frac{p(x)-p(0^+)}{x-0^+}=p'(\xi) \Rightarrow p(x)-1 \leq M_0 x \Rightarrow (p(x)-1)\log(x) \leq M_0 x  \left|\log(x)\right|,
\end{equation}
which yields $\lim_{x \to 0^+ }x^{p(x)-1} =1$. Therefore, we have
\begin{align}\label{eq.9}
\left|\varphi'(x)\right|\leq p^+ x^{p(x)-1}+\left|p'(x)\right| x^{p(x)}\left|\log(x)\right|\leq L_0,
\end{align}
where $L_0:=\delta^{p^-}\left(p^+ \delta^{-1}+M_0 \left|\log(\delta)\right|\right)$; that is, $\varphi'(x)$ remains uniformly bounded as $x \to 0^+$.\\
\textbf{Case} $1<x<\infty$: From $(\mathbf{p_2})$, there exist real numbers $M_{\infty}, R_{\infty}>0$ such that
\begin{equation}\label{eq.9a}
 p(x)-1 \leq \frac{M_{\infty}}{\log(x)}, \quad \forall x >R_{\infty}.
\end{equation}
Therefore, $\lim_{x \to \infty}x^{p(x)-1} \leq e^{M_{\infty}}$. Additionally, since
\begin{equation}\label{eq.9b}
0\leq\left|p'(x)\right| x^{p(x)}\log(x) \leq C_0 x^{-1-\alpha+p^+}\log(x),
\end{equation}
it reads $\lim_{x \to \infty}\left|p'(x)\right| x^{p(x)}\log(x)=0$. Thus, there exists $0\leq M_1<1$ such that
\begin{equation}\label{eq.9c}
\sup_{x \in (1,\infty)}\left|p'(x)\right| x^{p(x)}\log(x)\leq M_1 \quad  \text{whenever} \quad x >R_{\infty}.
\end{equation}
In conclusion, we have
\begin{equation}\label{eq.10}
\left|\varphi'(x)\right|\leq p^+ x^{p(x)-1}+\left|p'(x)\right| x^{p(x)}\log(x)\leq L_1,
\end{equation}
where $L_1:=p^+ e^{M_{\infty}}+M_1$, which means $\varphi'(x)$ stays uniformly bounded as $x \to \infty$. Therefore, $\varphi'(x)$ is uniformly bounded for any $x \in (0,\infty)$. Lastly, applying the MVT once more concludes that
\begin{equation}\label{eq.11}
\left|\varphi(x)-\varphi(y)\right|\leq L \left|x-y\right|, \quad \text{where} \quad L:=\max\{L_0,L_1\}.
\end{equation}
$(ii)$ Note that if $0<x\leq\delta\leq 1$, then
\begin{equation}\label{eq.12}
x^{p(x)}\leq \delta^{p^-}\leq \delta^{p^-}(1+x).
\end{equation}
For the case $1<x<\infty$: using the same reasoning as with obtaining \eqref{eq.9a}, it reads
\begin{equation}\label{eq.13}
e^{(p(x)-1)\log (x)}\leq e^{M_{\infty}}
\end{equation}
which yields
\begin{equation}\label{eq.15}
x^{p(x)}\leq e^{M_{\infty}}(1+x).
\end{equation}
Therefore, for any $x \in (0,\infty)$
\begin{equation}\label{eq.16}
x^{p(x)}\leq K(1+x), \quad \text{where} \quad K:= \max\{\delta^{p^-}, e^{M_{\infty}}\}.
\end{equation}
\end{proof}

\begin{theorem}\label{Thrm:3.1}
Assume that $p(\cdot) \in \mathcal{S}$. Assume also that the initial value $X(0)=x_0$ has a finite second moment: $\mathbb{E}[x_0^2]<\infty$, and is independent of $\{W(t),\, t\geq 0\}$. Then there exists a unique strong solution $X(t)$ of \eqref{eq.1} in $\mathcal{M}^{2}[0,T]$ with continuous
paths. Further,
\begin{equation}\label{eq.18}
\mathbb{E}\left[\sup_{t \leq T}|X(t)|^{2}\right] \leq (1+3\mathbb{E}[x_0^2])e^{3\mu^2T^2+24\sigma^2TK^2}.
\end{equation}
\end{theorem}
\begin{proof} Since the drift, $\mu x$, and the diffusion, $\sigma x^{p(x)}$, terms of \eqref{eq.1} satisfy the Lipschitz condition \eqref{eq.7a} and the linear growth condition \eqref{eq.7b}, there exists a unique solution $X(t)$ of \eqref{eq.1} in $\mathcal{M}^{2}[0,T]$ (see, e.g., \cite{klebaner12,oksendal03}). Next, we proceed to verify the inequality \eqref{eq.18}. To do so, we will adopt the approach used in Lemma 3.2 of \cite{mao07}. For every integer $n\geq 1$, define the stopping time
\begin{equation}\label{eq.18ca}
\tau_n=T \wedge \inf\{ t \in [0,T]: |X(t)|\geq n\}.
\end{equation}
Then, $\tau_n \uparrow T$ a.s. Set $X_n(t)=X(t \wedge \tau_n)$ for $t \in [0,T]$. Applying the H\"{o}lder inequality provides
\begin{align} \label{eq.18a}
|X_n(t)|^{2}& \leq 3x_0^2+ 3\mu^2\bigg|\int_0^t X_n(s)ds\bigg|^2 +3\bigg|\int_0^t X_n(s)^{p(X_n(s))}dW(s)\bigg|^2\nonumber \\
&\leq 3x_0^2 + 3\mu^2T \int_0^t (1+|X_n(s)|^2)ds + 3\bigg|\int_0^t X_n(s)^{p(X_n(s))}dW(s)\bigg|^2.\nonumber \\
\end{align}
Notice that from \eqref{eq.12}, \eqref{eq.15}, and the fact that $X(t) \in \mathcal{M}^{2}[0,T]$,  it follows that
\begin{equation}\label{eq.18bb}
\mathbb{E}\left[\int_{0}^{t}|X(s)|^{2p(X(s))}ds\right]\leq 2(t+1)\left(\delta^{2p^-}+e^{2M_{\infty}}\right)\mathbb{E}\left[\int_{0}^{t}X(s)^2ds \right]<\infty.
\end{equation}
Thus, $X(t)^{p(X(t))} \in \mathcal{M}^{2}[0,T]$.\\
Next, taking the expectation, applying the Burkholder-Davis-Gundy (BDG) inequality, the result \eqref{eq.18bb}, and the linear growth condition provides
\begin{align} \label{eq.19}
\mathbb{E}\left[\sup_{s \leq t}|X_n(s)|^{2}\right]& \leq 3\mathbb{E}[x_0^2] + 3\mu^2T \int_0^t (1+\mathbb{E} |X_n(s)|^2)ds\nonumber \\
& +3\sigma^2\mathbb{E}\left[\sup_{0 \leq s \leq t}\bigg|\int_0^s X_n(r)^{p(X(r))}dW(r)\bigg|^2\right]\nonumber \\
&\leq 3\mathbb{E}[x_0^2] + 3\mu^2T \int_0^t (1+\mathbb{E}|X_n(s)|^2)ds + 12\sigma^2 \mathbb{E} \left[\int_0^t |X_n(s)|^{2p(X(s))} ds\right] \nonumber \\
& \leq 3\mathbb{E}[x_0^2] + (3\mu^2T+24\sigma^2K^2) \int_0^t (1+\mathbb{E}|X_n(s)|^2) ds.
\end{align}
Therefore,
\begin{align} \label{eq.20}
1+\mathbb{E}\left[\sup_{s \leq t}|X_n(s)|^2\right]& \leq 1+ 3\mathbb{E}[x_0^2] + (3\mu^2T+24\sigma^2K^2) \int_0^t \left(1+\mathbb{E}\left[\sup_{r\leq s}|X_n(r)|^2\right] \right) ds.
\end{align}
Now, applying the Gronwall inequality yields
\begin{equation}\label{eq.21}
\mathbb{E}\left[\sup_{t \leq T}|X_n(t)|^{2}\right] \leq  (1+3\mathbb{E}[x_0^2])e^{3\mu^2T^2+24\sigma^2TK^2}.
\end{equation}
Thus,
\begin{equation}\label{eq.21ac}
\mathbb{E}\left[\sup_{t \leq \tau_n}|X_n(t)|^{2}\right] \leq  (1+3\mathbb{E}[x_0^2])e^{3\mu^2T^2+24\sigma^2TK^2}.
\end{equation}
Lastly, letting $n \to \infty$ gives \eqref{eq.18}.
\end{proof}
Next, we provide an error bound which can be used to compare our model \eqref{eq.1} and the geometric Brownian motion analytically.
\begin{theorem}\label{Thrm:3.2}
Assume that $X(t)$ is the solution  of \eqref{eq.1}, and $Y(t)$ is the solution  of the  geometric Brownian motion
\begin{align}\label{eq.26}
dY(t) &= \mu Y(t)dt + \sigma Y(t)dW(t),\quad t\in [0,T], \nonumber \\
Y(0) &= x_0.
\end{align}
If $p(\cdot) \in \mathcal{S}$, then it holds
\begin{equation}\label{eq.27}
\mathcal{E}=\mathbb{E}\left[\sup_{t \leq T}|X(t)-Y(t)|\right] \leq \sqrt{12\sigma^2e^{3T^2\mu^2 + 12T\sigma^2}}\, \Lambda \sup_{x >0}\left|p(x)-1\right|,
\end{equation}
where $\Lambda=\left|\log(\lambda)\right|\,(\lambda+\lambda^{p^+})+\left|\log(R)\right|\,(R+R^{p^+})$; $\lambda\in (0,1)$ and $R>1$ are real numbers.
\end{theorem}
\begin{proof}
Using \eqref{eq.1} and \eqref{eq.26}
\begin{align} \label{eq.28}
d\Delta_t &= \mu \Delta_t dt + \sigma \left(X(t)^{p(X(t))}-Y(t)\right)dW(t)\nonumber \\
&=\mu \Delta_t dt + \sigma \Delta_t dW(t) + \sigma \left(X(t)^{p(X(t))}-X(t)\right)dW(t)
\end{align}
where $\Delta_t:=(X(t)-Y(t))$,\,  $t\in [0,T]$.\\ Then
\begin{align} \label{eq.29}
\mathbb{E}\left[\sup_{s \leq t}|\Delta_s|^2\right] & \leq 3\mu^2 \mathbb{E}\left[\sup_{s \leq t} \left|\int_{0}^{s} \Delta_r dr\right|^{2}\right]+ 3\sigma^2 \mathbb{E}\left[\sup_{s \leq t} \left|\int_{0}^{s} \Delta_r dW(r)\right|^{2}\right] \nonumber \\
&+3\sigma^2 \mathbb{E}\left[\sup_{s \leq t} \left|\int_{0}^{s} \left(X(r)^{p(X(r))}-X(r)\right)dW(r)\right|^{2}\right].
\end{align}
Applying the BDG and H\"{o}lder inequalities gives
\begin{align} \label{eq.30}
\mathbb{E}\left[\sup_{s \leq t}|\Delta_s|^2\right] & \leq (3t\mu^2 + 12\sigma^2) \int_{0}^{t} \mathbb{E}\left[\sup_{r \leq s}|\Delta_r|^2\right] ds \nonumber \\
&+12\sigma^2 \mathbb{E} \int_{0}^{t} \mathbb{E}\left[\sup_{r \leq s}\left|X(r)^{p(X(r))}-X(r)\right|^{2}\right]ds.
\end{align}
Now we can apply the Gronwall inequality
\begin{align} \label{eq.31}
\mathbb{E}\left[\sup_{s \leq t}|\Delta_s|^2\right] & \leq 12\sigma^2 e^{3T^2\mu^2 + 12T\sigma^2} \, \int_{0}^{t} \mathbb{E}\left[\sup_{r \leq s}\left|X(r)^{p(X(r))}-X(r)\right|^{2}\right]ds.
\end{align}
Note that since $X(t)$ and $Y(t)$ have continuous paths (Kolmogorov’s continuity theorem); that is,  for each fixed $\omega \in \Omega$, $t \longrightarrow X(t,\omega), Y(t,\omega)$ are continuous on the compact interval $[0,T]$, each assumes a finite maximum on $[0,T]$.
Thus, for $\omega \in \Omega$, there are random numbers $M_X(\omega), M_Y(\omega)>0$ such that
\begin{align} \label{eq.32}
\sup_{t\in [0,T]}X(t,\omega)= M_X(\omega) \quad \text{and} \quad \sup_{t\in [0,T]}Y(t,\omega)= M_Y(\omega).
\end{align}
Now, we define two subspaces of the state-space  $\Omega_{\infty}:=(0,\infty)$  of $X(t)$ and $Y(t)$ as follows:
\begin{align} \label{eq.33}
\Omega_{\lambda}=\left\{\omega \in \Omega:\, \sup_{t\in [0,T]}X(t),\, \sup_{t\in [0,T]}Y(t)\in (0,\lambda],\,\, \forall t \leq T,\,\,\lambda\in (0,1) \right\},
\end{align}
\begin{align} \label{eq.34}
\Omega_{R}=\left\{\omega \in \Omega:\, \sup_{t\in [0,T]}X(t),\, \sup_{t\in [0,T]}Y(t)\in [\lambda, R],\,\, \forall t \leq T,\,\,R>1 \right\}.
\end{align}
Next, we will give an approximation for the term
\begin{align} \label{eq.34a}
\left|X(\cdot)^{p(X(\cdot))}-X(\cdot)\right|=\left|x^{p(x)}-x\right|,\,\, x>0.
\end{align}
Set $u(x)=(p(x)-1)\log(x)$. Then
\begin{align} \label{eq.35}
\left|x^{p(x)}-x\right|=x\left|e^{u(x)}-1\right|.
\end{align}
Now, we treat $u$ as just a variable, and observe the behaviour of the scalar function $f(u)=e^u$ as $u$ varies. Then, without loosing the generality, we may assume that $u>0$ and apply the MVT to the restriction of $f$ to the interval $[0,u]$ for a point $u_0 \in (0,u)$ at which
\begin{align} \label{eq.35a}
e^{u}-1=u e^{u_0}.
\end{align}
Since $u_0 \in (0,u)$, there exists $\phi \in (0,1)$ such that $u_0=\phi u$. Now, if we let $u=u(x)$, then $\phi=\phi(x)\in (0,1)$, and hence, $u_0(x)=\phi(x) u(x)$. Thus
\begin{align} \label{eq.36}
e^{u(x)}-1 = u(x)e^{\phi(x) u(x)}.
\end{align}
Therefore, we have
\begin{align} \label{eq.37}
\left|x^{p(x)}-x\right|=\left|p(x)-1\right|\,\left|\log(x)\right|\,x\,x^{\phi(x)(p(x)-1)}\leq \left|p(x)-1\right|\,\left|\log(x)\right|(x+x^{p^+})
\end{align}
Thus
\begin{equation}\label{eq.40}
\left|\log(x)\right|(x+x^{p^+})\leq
\left\{\begin{array}{ll}
\left|\log(\lambda)\right|\,(\lambda+\lambda^{p^+}),\quad & \omega \in \Omega_{\lambda}, \\
\left|\log(R)\right|\,(R+R^{p^+}),\quad & \omega \in \Omega_{R}.
\end{array}\right.
\end{equation}
Putting all these together gives
\begin{align} \label{eq.41}
\sup_{t \leq T}\left|X(t)^{p(X(t))}-X(t)\right|=\sup_{x >0}\left|x^{p(x)}-x\right| \leq \Lambda\sup_{x >0}\left|p(x)-1\right|,
\end{align}
and hence
\begin{align} \label{eq.42}
\mathbb{E}\left[\sup_{t \leq T}\left|X(t)^{p(X(t))}-X(t)\right|^{2}\right] \leq \Lambda^{2}\, \sup_{x >0}\left|p(x)-1\right|^{2}.
\end{align}
Rewriting \eqref{eq.31} gives
\begin{align} \label{eq.43}
\mathbb{E}\left[\sup_{s \leq t}|\Delta_s|^2\right]  & \leq 12\sigma^2\Lambda^{2} e^{3T^2\mu^2 + 12T\sigma^2} \sup_{x >0}\left|p(x)-1\right|^{2}.
\end{align}
Now applying the Cauchy-Schwarz inequality provides the desired result
\begin{align} \label{eq.44}
\mathbb{E}\left[\sup_{s \leq t}|\Delta_s|\right] & \leq \sqrt{12\sigma^2\Lambda^{2} e^{3T^2\mu^2 + 12T\sigma^2}}\, \sup_{x >0}\left|p(x)-1\right|.
\end{align}
Now, define the event
\begin{align} \label{eq.45}
\Omega^c_{\infty}=\Omega_{\infty}\setminus (\Omega_{\lambda}\cup \Omega_{R})=\left\{\omega \in \Omega:\, \sup_{t\in [0,T]}X(t)>R \text{ or } \sup_{t\in [0,T]}Y(t)>R \right\},
\end{align}
and the stopping time
\begin{equation} \label{eq.46}
\tau=\inf\{t\geq0:\, \sup_{t\in [0,T]}X(t)>R \text{ or } \sup_{t\in [0,T]}Y(t)>R\}.
\end{equation}
We claim that $\tau \to \infty$ as $R \to \infty$.
In doing so, using the Markov inequality and \eqref{eq.21} it reads
\begin{align} \label{eq.47}
\mathbb{P}\left\{\omega \in \Omega:\, \sup_{t\in [0,T]}X(t)> R \right\} &\leq \frac{1}{R^2}\,\mathbb{E}\left[\sup_{t\in [0,T]}X(t)^2\right] \nonumber \\
&\leq \frac{1}{R^2}\,(1+3\mathbb{E}[x_0^2])e^{3\mu^2T^2+12\sigma^2TK^2}<\infty.
\end{align}
Employing the Borel-Cantelli lemma leads to
\begin{align} \label{eq.48}
\mathbb{P}\left\{\limsup_{R \to \infty}\left\{\omega \in \Omega:\, \sup_{t\in [0,T]}X(t)> R \right\} \right\}=0.
\end{align}
Since $Y(t)$ is the solution  of the  geometric Brownian motion \eqref{eq.26}, it has a finite second moment for all $t \geq 0$, more precisely
$$
\mathbb{E}\left[\sup_{t\in [0,T]}Y(t)^2\right]=\mathbb{E}\left[Y(t)^2\right]=\mathbb{E}[x_0^2] e^{(2\mu+\sigma^2)t}<\infty.
$$
Therefore, applying the same steps as before gives
\begin{align} \label{eq.49}
\mathbb{P}\left\{\limsup_{R \to \infty}\left\{\omega \in \Omega:\, \sup_{t\in [0,T]}Y(t)> R \right\} \right\}=0.
\end{align}
In conclusion,
\begin{align} \label{eq.50}
\mathbb{P}\left\{\limsup_{R \to \infty}\Omega^c_{\infty} \right\}\leq & \left( \mathbb{P}\left\{\limsup_{R \to \infty}\left\{\omega \in \Omega:\, \sup_{t\in [0,T]}X(t)> R \right\} \right\} \right. \nonumber \\
&\left.+ \mathbb{P}\left\{\limsup_{R \to \infty}\left\{\omega \in \Omega:\, \sup_{t\in [0,T]}Y(t)> R \right\} \right\}\right)=0,
\end{align}
which concludes that $\tau \to \infty$ as $R \to \infty$.
\end{proof}

\section{Error estimates}\label{Sec4}

In this section, we test the accuracy of the upper-bound approximation provided in Theorem \ref{Thrm:3.2} via analytical and numerical error estimates.

\subsection{Analytical error estimates}\label{Sec4.1}
We would like to mention that the main purpose of the analytical error estimates is to provide a rough understanding of the suggested model's convergence-divergence dynamics rather than providing a tight estimate of the error itself. Having said that, the analytical results obtained here, Table \ref{table:1}, are consistent with those obtained in the actual errors observed in the simulations, Table \ref{table:2}, which shows the success of the proposed model.\\

We estimate the localized pathwise errors $\mathcal{E}(\lambda, R)$ on the interval $[\lambda,R]$ using
\begin{equation}\label{eq.51}
\mathcal{E}(\lambda, R)\leq \sqrt{12\sigma^2e^{3T^2\mu^2 + 12T\sigma^2}}\Lambda \sup_{x \in [\lambda, R]}\left|p(x)-1\right|
\end{equation}
with the fixed parameters
\begin{equation}\label{eq.51a}
\mu = 0.05,\quad \sigma = 0.2, \quad T = 1.
\end{equation}
Then the fixed coefficient in \eqref{eq.51} becomes
\begin{equation}\label{eq.54}
\sqrt{12\sigma^2 e^{3T^2\mu^2 + 12T\sigma^2}} = \sqrt{12 \cdot 0.04 \cdot 1.6279} = \sqrt{0.7814} \approx 0.8840.
\end{equation}
We  use the following variable exponent functions
\begin{equation}\label{eq.52}
p_1(x)=1+0.005e^{-0.1x},\qquad
p_2(x)=1+\frac{10^{-3}}{1+x}.
\end{equation}
Obviously $p_1(\cdot), p_2(\cdot) \in \mathcal{S}$. On each interval $[\lambda,R]$, the deviations are monotone in $x$, and hence, the localized suprema becomes
\begin{equation}\label{eq.53}
\sup_{x\in[\lambda,R]}\left|p_1(x)-1\right|=0.005e^{-0.1\lambda}, \quad
\sup_{x\in[\lambda,R]}|p_2(x)-1|=\frac{10^{-3}}{1+\lambda}.
\end{equation}
The Table 1 summarizes the error bounds obtained from \eqref{eq.51} for each $[\lambda,R]$ pair and $p_i(\cdot)$,\, $i=1,2$.
\begin{table}
\begin{center}
\begin{tabular}{ |c|c|c|c|c| }
\hline
Case & $\lambda$ & $R$ & $\mathcal{E}(\lambda, R)$ vs $p_1(x)$ & $\mathcal{E}(\lambda, R)$ vs $p_2(x)$ \\
\hline
1  & 0.1    & 1.1 & 0.002922 & 0.000538  \\
2  & 0.01   & 1.2 & 0.002335 & 0.000463  \\
3  & 0.001  & 1.4 & 0.004228 & 0.000844  \\
4  & 0.0001 & 1.5 & 0.005390 & 0.001077  \\
5  & 0.00001& 1.7 & 0.007986 & 0.001596  \\
6  & 0.0001 & 1.5 & 0.005390 & 0.001077  \\
7  & 0.001  & 1.4 & 0.004228 & 0.000844  \\
8  & 0.01   & 1.3 & 0.003416 & 0.000678  \\
9  & 0.1    & 1.2 & 0.003920 & 0.000721  \\
10 & 0.2    & 1.1 & 0.003687 & 0.000628  \\
\hline
\end{tabular}
\caption{Summary of analytical error estimates}
\label{table:1}
\end{center}
\end{table}
As it can be seen in the results, the error bound is highly sensitive to the choice of the interval $[\lambda, R]$; the error bound increases as $\lambda$ decreases ($\lambda \to 0^+$) and $R$ increases, which is expected since $\Lambda$ includes terms $\left|\log(\lambda)\right|$ and $\left|\log(R)\right|$. This shows that the error estimate formula \eqref{eq.51} is most effective when $p(\cdot)$ stays within a specific range away from zero. Notice also that $\mathcal{E}(\lambda, R) \to 0$ as $p(x) \to 1$,  as expected.

\subsection{Numerical error estimates}\label{Sec4.2}
In this subsection, we test our model \eqref{eq.1} against the GBM model with two alternative variable exponent functions $p_1(\cdot), p_2(\cdot) \in \mathcal{S}$. Each variable exponent function is paired with the GBM model in sample‐path and distributional comparisons. The results, visualized in Figures \ref{fig1:cir_p1}, \ref{fig2:cir_p2} and \ref{fig3:cir_p3}, support the theoretical validity of the model as a meaningful generalization of the GBM model.
\subsubsection{Experiment design}
\noindent We compare the solution of the GBM model
against
\begin{equation*}
\mathrm{\eqref{eq.1}}_i
\quad dX(t) = \mu X(t)dt + \sigma X(t)^{p_i(X(t))}dW(t)),
\quad i=1,2,
\end{equation*}
with the variable exponent functions $p_i(\cdot)$ given in \eqref{eq.52}. \\
The numerical experiments were conducted with the following parameters:
\begin{itemize}
\item $\mu = 0.05$, $\sigma = 0.2$, $T = 1$, $\Delta t = 0.001$, $X(0) = x_0 = 1.0$.
\item Number of time steps: 1000.
\item Number of sample paths: 20000 (10000 main + 10000 antithetic variates).
\end{itemize}
The log-transformed Milstein scheme was used along with antithetic variates. Identical Brownian increments were used across all models to isolate the impact of the variable exponent functions $p_i(\cdot)$. The pathwise strong convergence errors between our model  and the GBM model, namely
$$
\mathcal{E}=\mathbb{E}\left[\sup_{t\in[0,T]}|X(t)_{\mathcal{P}}-X(t)_{\text{GBM}}|\right]
$$
were computed for each variable exponent function. In Figures \ref{fig1:cir_p1} and \ref{fig2:cir_p2}, the left panel shows single sample path trajectories using identical Brownian increments, whereas the right panel displays terminal value distributions from 20,000 Monte Carlo simulations.

\subsubsection{Analysis}
The numerical experiments clearly demonstrate that the variable exponent framework suggested by \eqref{eq.1} provides considerable advantages over the GBM model in terms of coffering a flexible framework for volatility modeling while maintaining the essential characteristics of the GBM model almost identically. We strongly believe that the model can be used modeling real-world phenomena across different disciplines. For example,  Figure \ref{fig3:cir_p3} shows a financial application of implied volatility: comparison of implied volatility patterns shows $p_1(\cdot)$  generates a realistic volatility smile while $p_2(\cdot)$  remains nearly flat across strikes. The contrasting behaviors show the model's flexibility across the different regimes: $p_1(\cdot)$  captures market realities with significant state-dependence, while $p_2(\cdot)$  converges to GBM behavior since $p_2(\cdot)$ deviates from $1$ by only $\approx 10^{-3}$.
We'd also like to mention that the model is computationally tractable using standard discretization methods, such as the log-transformed Milstein or Euler-Maruyama schemes, both of which ensure its numerical stability and efficiency.

\begin{figure}[htbp!]
  \centering
  \includegraphics[width=1.05\textwidth]{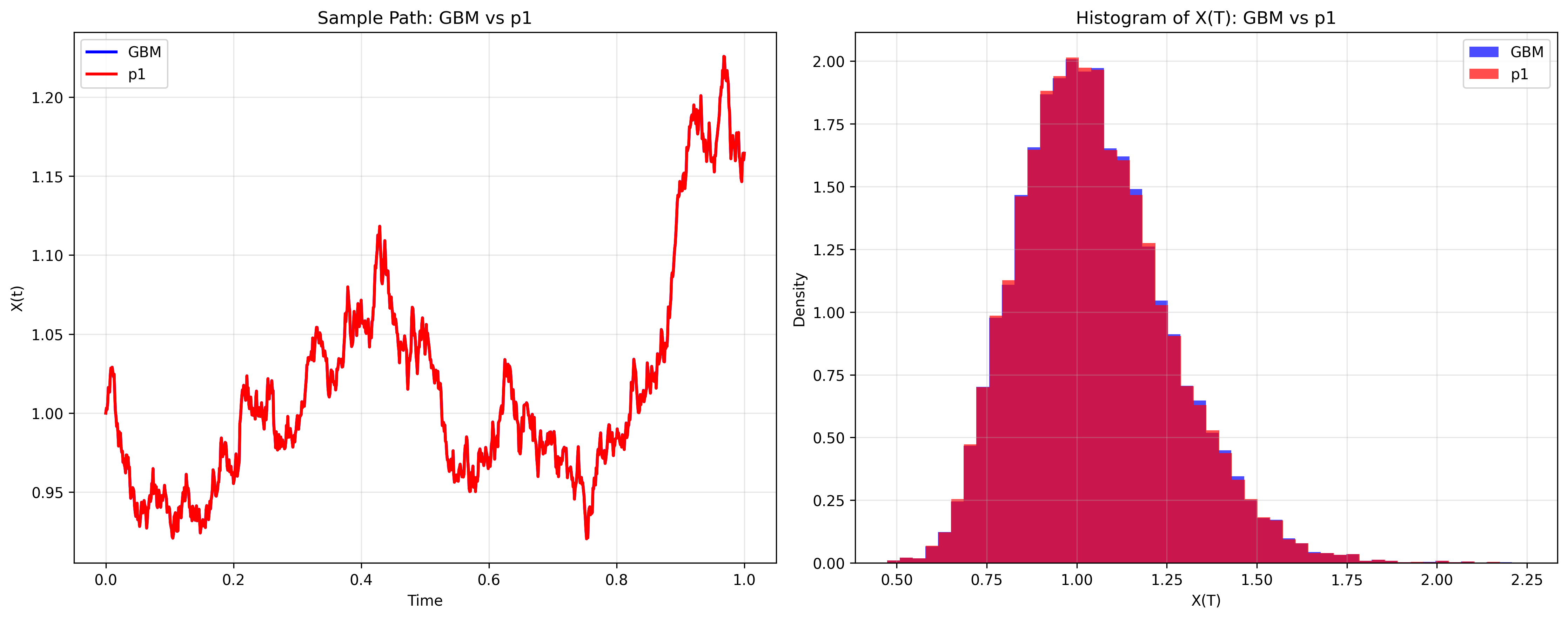}
  \caption{$p_1(\cdot)$ vs GBM.}  \label{fig1:cir_p1}
\end{figure}
\begin{figure}[htbp!]
  \includegraphics[width=1.05\textwidth]{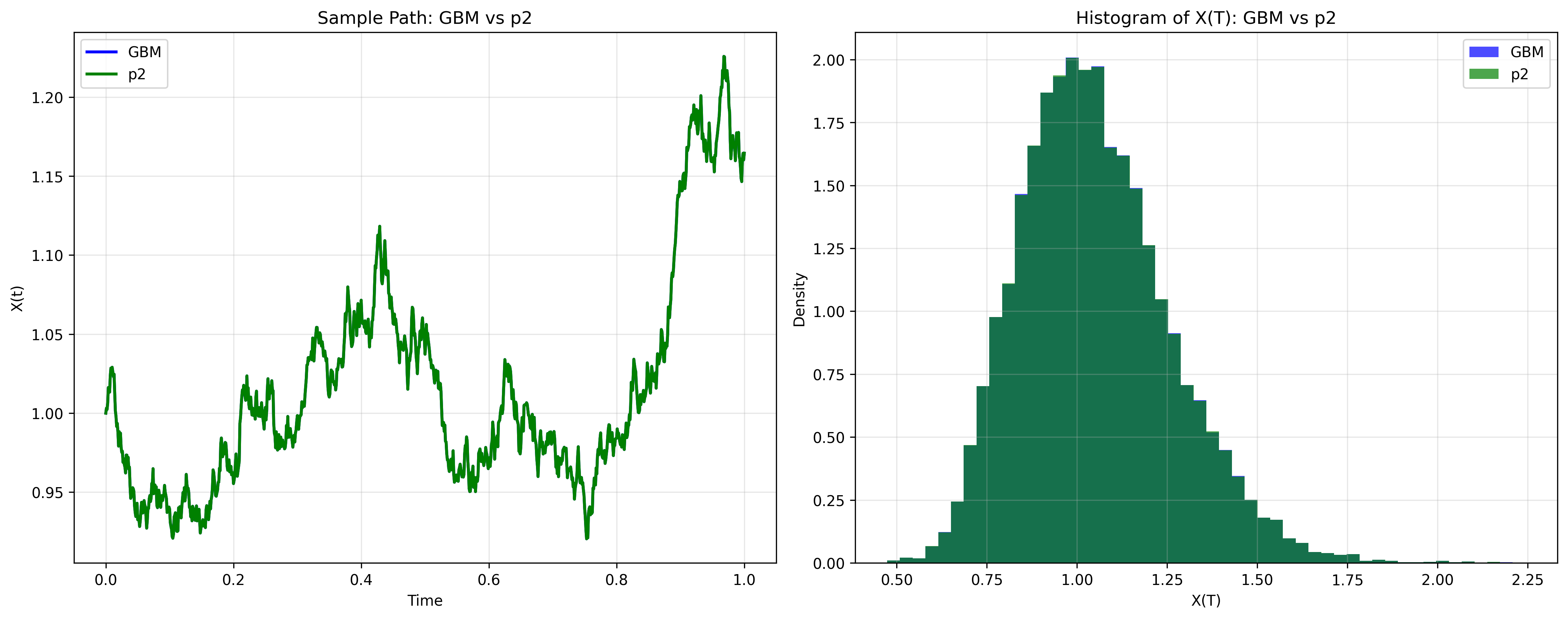}
  \caption{$p_2(\cdot)$ vs GBM.}  \label{fig2:cir_p2}
\end{figure}
\begin{figure}[htbp!]
  \includegraphics[width=1.07\textwidth]{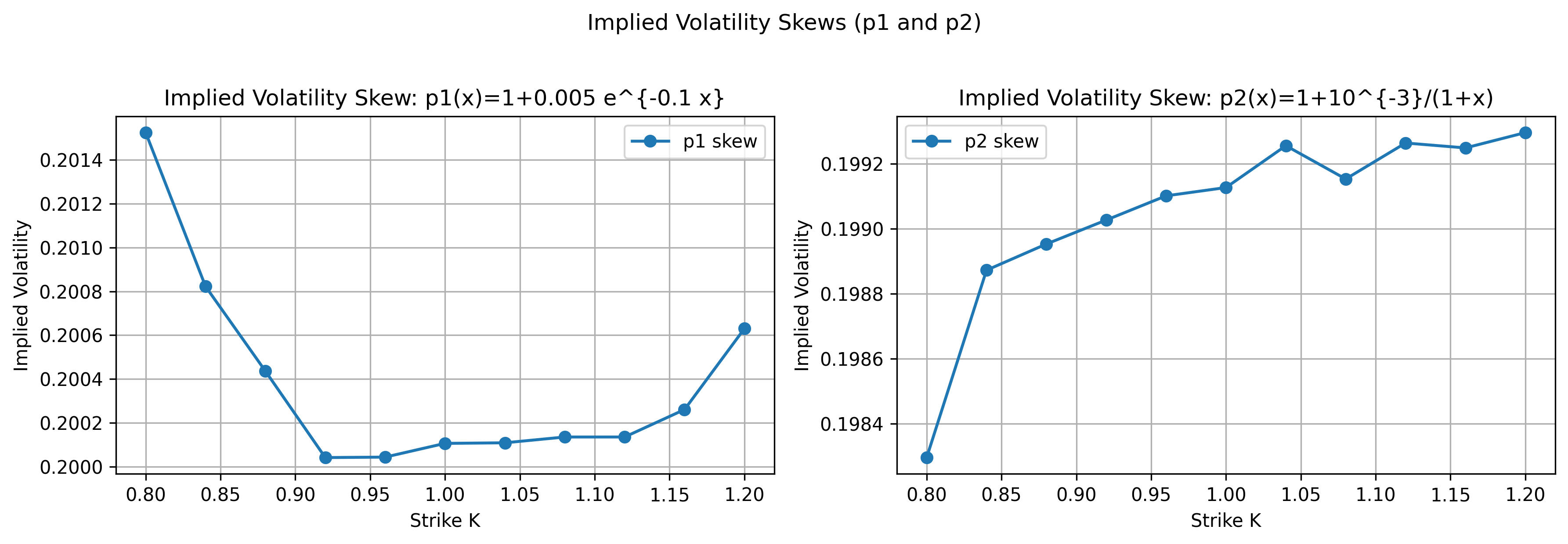}
  \caption{Implied-volatility skews at maturity T=1 estimated from Monte Carlo call prices.} \label{fig3:cir_p3}
\end{figure}

\begin{table}
\begin{center}
\begin{tabular}{ |c|c|c| }
\hline
Function & Volatility Range & $\mathcal{E}=\mathbb{E}\left[\sup_{t\in[0,T]}|X(t)_{\mathcal{P}}-X(t)_{\text{GBM}}|\right]$ \\
\hline
$p_1(x)$ & [0.0198, 1.0049]   & 0.000145  \\
$p_2(x)$ & [0.0200, 1.0003]   & 0.000015  \\
\hline
\end{tabular}
\caption{Summary of numerical error estimates}
\label{table:2}
\end{center}
\end{table}

\newpage
\section{Conclusions}\label{sec5}

The introduced model \eqref{eq.1} offers a mathematically sound and computationally feasible approach to modeling systems where volatility structure depends on the underlying state variable. This new variable exponent framework provides significant improvements over classical constant-volatility models. Choosing the form of $p(\cdot)$ enables precise calibration to empirical features or theoretical constraints. The analytical error estimates (Table \ref{table:1}) are consistent with the actual errors observed in simulations (Table \ref{table:2}) confirming that the proposed model represents a reasonable generalization of the classical GBM framework.

\section{Appendix} \label{Sec6}
In this section, we present both the Itô and Stratonovich interpretations of the problem and provide a detailed comparative analysis of their respective formulations.\\
In stochastic calculus, the It\^o and Stratonovich integrals represent two distinct interpretations of SDEs. Thus, we can present our problem by using either
the  It\^o form, i.e. \eqref{eq.1}, or the Stratonovich form
\begin{align}\label{eq.69}
dX(t) &= \mu X(t)dt + \sigma X(t)^{p(X(t))}\circ dW(t). \tag{{${\mathcal{ST}}$}}
\end{align}
Although both formulations are mathematically sound, they produce different outcomes when applied to models with state-dependent diffusion coefficients. Nevertheless, by employing the conversion formula, we can express the Itô counterpart \eqref{eq.70} of \eqref{eq.69} as follows
\begin{align}\label{eq.70}
dX(t) &= \left(\mu X(t) +  h(X(t))\right) dt + \sigma X(t)^{p(X(t))} dW(t), \tag{{${\mathcal{ST_I}}$}}
\end{align}
where
\begin{align}\label{eq.70a}
h(X(t))=\frac{1}{2} \xi^2 X(t)^{2p(X(t))} \left(p^{\prime}(X(t)) \log X(t) + \frac{p(X(t))}{X(t)}\right).
\end{align}
is the drift-correction term. Given that \eqref{eq.70} is mathematically equivalent to \eqref{eq.69}, it suffices to compare \eqref{eq.1} with \eqref{eq.70} to highlight the differences in the resulting behavior.\\
We begin by noting that \eqref{eq.70} is well-defined. Specifically, if the process $X(t)$ satisfies \eqref{eq.70}, then under the assumptions of Lemma \ref{Lem:2.1}, $X(t)$ remains strictly positive for all $t \in [0,T]$. This can be established by applying Feller’s non-attainability test to \eqref{eq.70}, following an argument analogous to that used in Lemma \ref{Lem:2.1}. Indeed, let
\begin{align}\label{eq.71}
\mathcal{S}_{p}(x)&:=\mu x +\frac{1}{2}\sigma^2 x^{2p(x)}\left(p^{\prime}(x)\log(x)+\frac{p(x)}{x}\right)\nonumber \\
&-\sigma^2 x^{2p(x)}\left(p^{\prime}(x)\log(x)+\frac{p(x)}{x}\right).
\end{align}
Then,
\begin{align}\label{eq.71a}
\mathcal{S}_{p}(x) \geq \mu x-\frac{1}{2}\sigma^2 \left(|p^{\prime}(x)||\log(x)|x^{2p(x)}+x^{2p(x)-1}p(x)\right),
\end{align}
which yields, for $0<x\leq\delta<1$,
\begin{align}\label{eq.72}
\lim_{x \to 0^+}\mathcal{S}_{p}(x)& \geq \lim_{x \to 0^+} \left[\mu x-\frac{1}{2}\sigma^2 \left(M_0|\log(x)|x^{2p(x)}+x^{2p(x)-1}p^+\right)\right]=0\geq 0.
\end{align}
Therefore, the diffusion process $X(t)$ satisfying \eqref{eq.70} never reaches zero a.s. when initiated in $(0,\infty)$. Consequently, \eqref{eq.70} is well-defined.\\
In what follows, in order to analyze moment dynamics and emphasize their structural differences, we derive the formal equations governing the evolution of the mean and variance for both models.\\
Assume $X(t)$ (with $X(0)>0$) be the solution to \eqref{eq.1} and $X_s(t)$ (with $X_s(0)>0$) be the solution to \eqref{eq.70}, with their respective means $\mathbb{E}X(t)$ and $\mathbb{E}X_s(t)$, $t \in [0,T]$. Now, taking expectations and using the zero mean property of the It\^{o} integral yields
\begin{align}\label{eq.73}
\frac{d}{dt}\mathbb{E}X(t) &= \mu \mathbb{E}X(t),
\end{align}
and
\begin{align}\label{eq.74}
\frac{d}{dt}\mathbb{E}X_s(t) &= \mu \mathbb{E}X_s(t) + \mathbb{E} h(X_s(t)).
\end{align}
Note that the sign of $\mathbb{E} h(X)$ is determined by the state-dependent function
\begin{align}\label{eq.74a}
 X \mapsto p^{\prime}(X) \log X+ \frac{p(X)}{X}=\frac{d}{dX}\left[p(X) \log X \right].
\end{align}
Thus,
\begin{equation}\label{eq.75}
\mathbb{E} h(X) \begin{cases} > 0 , & \text{if }\, p(X)\log X \text{ is strictly increasing}, \\ <0, & \text{if  }\, p(X)\log X \text{ is strictly decreasing}. \end{cases}
\end{equation}
Therefore, \eqref{eq.70} will yield a mean that differs from that of \eqref{eq.1}, producing either a lower or higher value. The two approaches will generate identical means only when $p(X)\log X$ remains constant, causing the models to be equivalent. This condition, however, is not the case under assumptions $(\mathbf{p_1})$-$(\mathbf{p_2})$.\\
Next, consider the second moments $\mathbb{E}X^2(t)$ and $\mathbb{E}X^2_s(t)$, $t \in [0,T]$, of \eqref{eq.1} and \eqref{eq.70},  respectively.
Applying the It\^o formula for \eqref{eq.1} and \eqref{eq.70}  yields
\begin{align}\label{eq.76}
dX^{2}(t) &= \left[2\mu X^2(t)+\sigma^{2}X(t)^{2p(X(t))}\right]dt + 2\sigma X(t)^{p(X(t))+1} dW(t),
\end{align}
and
\begin{align}\label{eq.77}
dX_s^{2}(t) &= \left[2\mu X_s^2(t)+\sigma^{2}X_s(t)^{2p(X_s(t))}+2X_s(t)h(X_s(t))\right]dt\nonumber \\
& + 2\sigma X_s(t)^{p(X_s(t))+1} dW(t),
\end{align}
respectively. Taking expectations gives
\begin{align}\label{eq.78}
\frac{d}{dt}\mathbb{E}X^2(t) &= \mathbb{E}\left[2\mu X^2(t)+\sigma^{2}X(t)^{2p(X(t))}\right],
\end{align}
and
\begin{align}\label{eq.79}
\frac{d}{dt}\mathbb{E}X^2_s(t) &= \mathbb{E}\left[2\mu X_s^2(t)+\sigma^{2}X_s(t)^{2p(X_s(t))}\right]+\mathbb{E}\left[2X_s(t)h(X_s(t))\right].
\end{align}
The second moment's rate of change under \eqref{eq.70}  includes an additional term \newline $\mathbb{E}\left[2X_s(t)h(X_s(t))\right]$. This additional contribution means that the two interpretations produce distinct second moment dynamics and, therefore, different variance behavior.

\begin{remark}
The selection between It\^o and Stratonovich interpretations carries significant practical implications beyond mere technical considerations, affecting both the behavior and analysis of stochastic models. The It\^o interpretation remains the standard choice due to its compatibility with martingale theory in option pricing and its superior analytical properties for establishing moment bounds and ensuring positivity under suitable conditions. Conversely, the Stratonovich formulation proves more appropriate for physical systems featuring parametric noise (i.e. $X(t)^{p(X(t))}$) or multiplicative noise structures, where it maintains important symmetries including the standard chain rule, despite the martingale property being less relevant in such contexts.
\end{remark}

\section*{Data usage statement}
\noindent All data used in the numerical experiments are generated by Monte Carlo simulation, no proprietary or external datasets were used. All figures and distributions obtained purely from simulating the stochastic differential equations studied in the paper by using Python programming language.

\section*{Acknowledgments}
This work was supported by Athabasca University Research Incentive Account [140111 RIA].
\section*{ORCID}
https://orcid.org/0000-0002-6001-627X

\end{document}